\documentclass[12pt]{amsart}
\usepackage{tikz}
\usepackage{hyperref}
\usepackage[margin=1.3in]{geometry}
\usetikzlibrary{matrix,arrows}

\usepackage{tabmac}

\newcommand\tabll[1]{{\tableau[pY]{#1}}}
\def\oPi{\mathring{\Pi}}
\def\sm{{\rm sm}}
\def\C{{\mathbb C}}
\def\R{{\mathbb R}}
\def\Gr{{\rm Gr}}

\def\tF{{\tilde F}}
\def\codim{{\rm codim}}
\def\Bound{{\rm Bound}}
\def\id{{\rm id}}
\def\oMat{\mathring{\Mat}}
\def\Mat{{\rm Mat}}
\def\inv{{\rm inv}}
\def\Z{{\mathbb Z}}
\def\pt{{\rm pt}}

\newcommand\defn[1]{{\bf #1}}
\newtheorem{theorem}{Theorem}
\newtheorem{remark}[theorem]{Remark}
\newtheorem{definition}[theorem]{Definition}
\newtheorem{lemma}[theorem]{Lemma}
\newtheorem{proposition}[theorem]{Proposition}
\newtheorem{corollary}[theorem]{Corollary}
\newtheorem{problem}[theorem]{Problem}

\newtheorem{example}[theorem]{Example}

\numberwithin{theorem}{section}
\begin{document}
\title{Amplituhedron cells and Stanley symmetric functions}
\author{Thomas Lam}
\address{Department of Mathematics, University of Michigan,
2074 East Hall, 530 Church Street, Ann Arbor, MI 48109-1043, USA}
\email{tfylam@umich.edu}\thanks{T.L. was supported by NSF grant DMS-1160726.}
\begin{abstract}
The amplituhedron was recently introduced in the study of scattering amplitudes in $N=4$ super Yang-Mills.  We compute the cohomology class of a tree amplituhedron subvariety of the Grassmannian to be the truncation of an affine Stanley symmetric function.
\end{abstract}

\maketitle
\section{Introduction} \label{sec:intro}
Let $\Gr(k,n)$ denote the Grassmannian of $k$-planes in $\C^n$.  It has a stratification by {\it positroid varieties} $\Pi_f$ \cite{Pos,KLS}, where $f$ ranges over the finite set $\Bound(k,n)$ of $(k,n)$-bounded affine permutations (defined in Section \ref{sec:main}).  Each positroid variety is the intersection of $n$ cyclically rotated Schubert varieties.  In \cite{KLS}, Knutson, Lam, and Speyer identified the cohomology class of a positroid variety with the {\it affine Stanley symmetric function} $\tF_f$ \cite{Lam}.

The totally nonnegative part $\Gr(k,n)_{\geq 0}$ of the real Grassmannian is the locus where all Pl\"ucker coordinates take nonnegative values \cite{Lus,Pos}, and was studied extensively by Postnikov.  Arkani-Hamed and Trnka \cite{AT}, motivated by the study of scattering amplitudes in $N=4$ super Yang-Mills, proposed that $\Gr(k,n)_{\geq 0}$ should be considered a Grassmannian-analogue of a simplex.  Arbitrary convex polytopes are images of simplices under affine or linear maps, and Arkani-Hamed and Trnka proposed to study the {\it amplituhedron}\footnote{In this paper we shall only consider the ``tree" amplituhedron, leaving the ``loop" amplituhedron for later work.}: the image of the totally nonnegative Grassmannian induced by a linear map $Z: \R^n \to \R^{k+m}$ (which in turn gives a rational map $Z_\Gr:\Gr(k,n) \to \Gr(k,k+m)$).   In addition, physical considerations suggested the study of triangulations of the amplituhedron, obtained as unions of  images of the positroid cells $(\Pi_f)_{\geq 0}:=\Pi_f \cap \Gr(k,n)_{\geq 0}$, again under the map $Z_\Gr$.  Specifically,  the {\it scattering amplitude} can be obtained by summing differential forms over cells of a triangulation of the amplituhedron.

The behavior of positroid cells under the map $Z_\Gr$ exhibit a number of features not present in usual convex geometry, including:
\begin{enumerate}
\item
Even when $Z$ is generic, the image $Z_\Gr((\Pi_f)_{\geq 0})$ may not have the expected dimension.  For example, even if $\dim((\Pi_f)_{\geq 0}) = \dim(\Gr(k,k+m))$ we may have $\dim(Z_\Gr((\Pi_f)_{\geq 0})) < \dim(\Gr(k,k+m))$ for generic $Z$.
\item
The map $Z_\Gr|_{((\Pi_f)_{\geq 0})}: ((\Pi_f)_{\geq 0}) \to Z_\Gr((\Pi_f)_{\geq 0})$ can be dimension-preserving, but have degree $d$ greater than one.
\end{enumerate}

In this paper, we study the complex geometry of the behavior of the stratification $\Gr(k,n) = \bigcup_f \Pi_f$ under the map $Z_\Gr$, from a Schubert calculus perspective.  Let $Y_f$ denote the closure of the image of $\Pi_f$ under $Z_\Gr$.  We call $Y_f$ an {\it amplituhedron variety} in the case that it has the same dimension as $\Pi_f$.  

Recall that the cohomology\footnote{Henceforth, we shall always take cohomologies with $\Z$-coefficients.} ring $H^*(\Gr(k,n),\Z)$ can be identified with a quotient of the ring of symmetric functions, and that the basis of Schubert classes correspond to the Schur functions $s_\lambda$, labeled by partitions $\lambda \subseteq (n-k)^k$ that fit inside a $k \times (n-k)$ rectangle.  Let $\ell = n-k-m$.  For $\mu \subseteq (m)^k$ we let $\mu^{+\ell} \subseteq (n-k)^k$ be the partition obtained from $\mu$ by adding $\ell$ columns of height $k$ to the left of $\mu$.  For example, with $\ell = 2$ and $k = 4$, we may have
$$
\mu = \tabll{&&&\\&&\\ &\bl \\ & \bl} \qquad \qquad \qquad
\mu^+ = \tabll{&&&&&\\&&&&\\&&\\&&} 
$$

Given $f = \sum_{\lambda \subset (n-k)^k} c_\lambda s_\lambda$ representing a cohomology class in $H^*(\Gr(k,n))$, we define the {\it truncation} $\tau_{k+m}(f) \in H^*(\Gr(k,k+m))$ by
$$
\tau_{k+m}(f) = \sum_{\mu \subseteq (m)^k} c_{\mu^{+\ell}} s_\mu.
$$
Let $d_f$ denote the degree of the map $Z_\Gr|_{\Pi_f}: \Pi_f \to Y_f$.

\begin{theorem}
The cohomology class of the amplituhedron variety $Y_f$ is equal to $\frac{1}{d_f} \tau_{k+m}(\tF_f)$.
\end{theorem}

We also prove that $\tau_{k+m}(\tF_f) = 0$ if and only if $\dim Y_f < \dim \Pi_f$.  As a corollary, we deduce a criterion for $Z_\Gr((\Pi_f)_{\geq 0})$ to have the same dimension as $(\Pi_f)_{\geq 0}$, which corresponds to the physical notion of ``kinematical support".  We also obtain some estimates on the degree $d_f$.  We present a number of possible further directions in Section \ref{sec:other}.
% including a reality conjecture reminiscent of the Shapiro-Shapiro conjecture.

\medskip

{\bf Acknowledgments.}  My understanding of the amplituhedron owes a lot to Nima Arkani-Hamed and Jara Trnka, who have taught me much through many conversations.   I thank Jake Bourjaily for discussions related to scattering amplitudes, and for his ``positroids'' package in Mathematica.  I benefited from discussions with Allen Knutson and Alex Postnikov, and am especially grateful to David Speyer for a number of helpful comments.

Much of what I understand about scattering amplitudes was learnt during a reading seminar with Henriette Elvang,  Yu-tin Huang, Cindy Keeler, Tim Olson, Sam Roland, and David Speyer.  I thank them all for teaching me this subject.

\section{Projection maps and Schubert varieties}
We fix positive integers $n,k,m$ satisfying $n > k+m$, and we let $\ell = n-k-m$.  Let $\Mat(n,k+m)$ denote the space of $n \times (k+m)$ matrices, and let $\oMat(n,k+m)$ denote the open subset of full rank $n \times (k+m)$ matrices.  We think of $Z \in \oMat(n,k+m)$ as a linear map $Z: \C^n \to \C^{k+m}$.  The map $Z$ induces a rational map $Z_\Gr: \Gr(k,n) \to \Gr(k,k+m)$ given by $X \mapsto X\cdot Z$.  The {\it exceptional locus} $E_Z$ of $Z_\Gr$ is the subset of $\Gr(k,n)$ where the map $Z_\Gr$ is not defined:
$$
E_Z = \{X \in \Gr(k,n)\mid X \cap \ker(Z) \neq (0)\}.
$$
Here $\ker(Z) \subset \C^n$ is the usual kernel of a linear map.  The exceptional locus $E_Z$ is in fact a Schubert variety which has codimension $m+1$.

\begin{lemma}\label{lem:bundle} The morphism $Z_\Gr :\Gr(k,n) \setminus E_Z \to \Gr(k,k+m)$ is a fiber bundle with fiber $\C^{k(n-k-m)}$.  
\end{lemma}
\begin{proof}
We use the $GL(n)$ actions on $Z$ and on $\Gr(k,n)$ to reduce to the case that $Z$ is the orthogonal projection of ${\rm span}(e_1,e_2,\ldots,e_n)$ onto ${\rm span}(e_1,e_2,\ldots,e_{k+m})$.  Then the map $Z_\Gr$ looks like
$$
(Y|*) \mapsto Y
$$
where $Y$ is a $k \times (k+m)$ matrix representing a point in $\Gr(k,k+m)$, and $*$ denotes the $\C^{k(n-k-m)}$ fiber.
\end{proof}

We use the notation $[a]:=\{1,2,\ldots,a\}$.
Let $I \in \binom{[n]}{k}$ be a $k$-element subset of $[n]$.  Let $F_\bullet = \{ 0 = F_0 \subset F_1 \subset \cdots F_{n-1} \subset F_n = \C^n\}$ be a flag in $\C^n$, so that $\dim F_i = i$.  The Schubert variety $X_I(F_\bullet)$ is given by
\begin{equation}\label{eq:schub}
X_I(F_\bullet) = \{X \in \Gr(k,n) \mid \dim(X \cap F_j) \geq \#(I \cap [n-j+1,n]) 
\text{ for all } j \in [n]\}.
\end{equation}
Thus $X_{[k]}(F_\bullet) = \Gr(k,n)$ and $\codim(X_I(F_\bullet)) = i_1+i_2+\cdots+i_k - (1+2+\cdots+k)$, where $I = \{i_1,i_2\ldots,i_k\}$.  Here and elsewhere, we always mean complex (co)dimension when referring to complex subvarieties.

Let $G_\bullet$ be a flag in $\C^{k+m}$.  Then $Z^{-1}(G_\bullet)$ is the partial flag
$$
Z^{-1}(G_\bullet):=\{\ker(Z) = Z^{-1}(G_0) \subset Z^{-1}(G_1) \subset \cdots \subset Z^{-1}(G_{k+m}) = Z^{-1}(\C^{k+m}) = \C^n\}
$$ 
of subspaces with successive dimensions $n-(k+m), n-(k+m)+1,\ldots,n$.  We denote by $Y_J(G_\bullet)$ a Schubert variety in $\Gr(k,k+m)$, where $J$ is a $k$-element subset of $[k+m]$.  A full-flag extension of $Z^{-1}(G_\bullet)$ is simply any flag $F_\bullet$ in $\C^n$ whose $n-(k+m), n-(k+m)+1,\ldots,n$-dimensional pieces give $Z^{-1}(G_\bullet)$. 

\begin{lemma}\label{lem:inverse}
We have
$$
\overline{Z_\Gr^{-1}(Y_I(G_\bullet))} = X_I(F_\bullet)
$$
where $F_\bullet$ is any full-flag extension of $Z^{-1}(G_\bullet)$, and $I \subset [k+m]$ is considered a subset of $[n]$ via the natural inclusion $[k+m] =\{1,2,\ldots,k+m\} \subset \{1,2,\ldots,n\}=[n]$. 
\end{lemma}
\begin{proof}
Suppose $X \in Z_\Gr^{-1}(Y_I(G_\bullet))$.  Then $\dim(Z_\Gr(X) \cap G_j) \geq \#(I \cap [k+m-j+1,k+m])$, and so $\dim(X \cap F_{j+n-k-m}) \geq \#(I \cap [k+m-j+1,n])$ for all $j \in [1,k+m]$.  That is, $\dim(X \cap F_{j'}) \geq \#(I \cap [n-j'+1,n])$ for all $j' \in [n-k-m+1,n]$.  Since membership in $X_I(F_\bullet)$ imposes no condition on $X \cap F_{j'}$ for $j' \in [1,n-k-m]$, we conclude that $ Z_\Gr^{-1}(Y_I(G_\bullet)) \subset X_I(F_\bullet)$.  But using Lemma \ref{lem:bundle}, we see that  $\overline{Z_\Gr^{-1}(Y_I(G_\bullet))}$ and $X_I(F_\bullet)$ are closed irreducible subvarieties of $\Gr(k,n)$ of the same dimension, and so must be identical.
\end{proof}

If $J=\{m+1,m+2,\ldots,k+m\}$ then $Y_J(G_\bullet)$ is a single point $Y = G_k \in \Gr(k,k+m)$.  Lemma \ref{lem:inverse} then says that $\overline{Z_\Gr^{-1}(Y)} = \Gr(k,Z^{-1}(Y))$ is a subGrassmannian of $\Gr(k,n)$.

\section{Cohomology class of a projection}
We shall need the following version of Kleiman transversality.

\begin{theorem}[{\cite[Corollary 4]{Kle}}] \label{thm:Kleiman}
Assume the base field is $\C$.  Let $X$ be an integral algebraic scheme with a transitive action of an algebraic group $G$.  Let $Y, Z \subset X$ be integral subschemes.   Then
\begin{enumerate}
\item 
There exists a dense subset $U \subset G$ such that for $g \in U$, the intersection $gY \cap Z$ is proper, that is, each component has dimension $\dim(Y) + \dim(Z) - \dim(X)$. 
\item
If in addition $Y$ and $Z$ are smooth, then $U$ can be chosen so that for all $g \in U$, the subschemes $gY$ and $Z$ intersect transversally, that is, the intersection $gY \cap Z$ is smooth and each component has dimension $\dim(Y) + \dim(Z) - \dim(X)$. 
\end{enumerate}
\end{theorem}

We remark that if $gY$ and $Z$ intersect transverally then the intersection $gY \cap Z$, being smooth, must be contained in the smooth locus of both $gY$ and $Z$.  We shall also need the following technical result which appears in the proof of Theorem \ref{thm:Kleiman}.

\begin{lemma}[{\cite[Lemma 1]{Kle}}] \label{lem:Kleiman}
Assume the base field is $\C$.  Consider a diagram with integral algebraic schemes:
\begin{center}
\begin{tikzpicture}[description/.style={fill=white,inner sep=2pt}]
\matrix (m) [matrix of math nodes, row sep=3em,
column sep=2.5em, text height=1.5ex, text depth=0.25ex]
{  & W & & Z \\
S& & X & \\ };
%\draw[double,double distance=5pt] (m-1-1) – (m-1-3);
\draw[->,font=\scriptsize]
(m-1-2) edge node[description] {$p$} (m-2-1);
\draw[->,font=\scriptsize]
(m-1-2) edge node[description] {$q$} (m-2-3);
\draw[->,font=\scriptsize]
(m-1-4) edge node[description] {$r$} (m-2-3);
\end{tikzpicture}
\end{center}
\begin{enumerate}
\item
Assume $q$ is flat.  Then, there exists a dense open subset $U$ of $S$ such that for each point $s \in U$, either the fibered product, $p^{-1}(s) \times_X Z$, is empty or it is equidimensional and its dimension is given by the formula,
$$
\dim(p^{-1}(s) \times_X Z) = \dim(p^{-1}(s)) + \dim(Z) - \dim(X).
$$
\item
Assume $q$ is flat with smooth fibers.  Assume $Z$ is smooth.  Then $p^{-1}(s) \times_X Z$ is smooth for each point $s$ in an open dense subset of $S$.
\end{enumerate}
\end{lemma}

%Let $E_Z = \{X \in \Gr(k,n) \mid X \cap \ker(Z) \neq (0)\}$ denote the exceptional locus for the map $Z_\Gr$.  Here $\ker(Z) \subset \C^n$ is the kernel of the map $Z: \C^n \to \C^{k+m}$.   
Let $W \in \Gr(k,n)$ be an irreducible subvariety.  For $Z \in \oMat(n,k+m)$, we define
$$
\tau_Z(W):= \overline{Z_\Gr(W \setminus E_Z)}.
$$
For a generic $Z$, the subscheme $W \setminus E_Z$ is irreducible and dense in $W$.  Thus $\tau_Z(W)$ is itself an irreducible subvariety.  There is a $GL(n)$-action on $\Gr(k,n)$ and a $GL(n)$-action on $\oMat(n,k+m)$.  We choose compatible conventions so that $gE_{Z_0} = E_{gZ_0}$. 

Let $Y_I = Y_I(G_\bullet) \subset \Gr(k,k+m)$ be a Schubert subvariety.  For $Z \in \oMat(n,k+m)$ a full-rank matrix, let $X^Z_I = \overline{Z_\Gr^{-1}(Y_I)} \subset \Gr(k,n)$ be as in Lemma \ref{lem:inverse}.

\begin{lemma}\label{lem:main}
Fix $W \subset \Gr(k,n)$ an irreducible subvariety, $Y_I \subset \Gr(k,k+m)$ a Schubert variety, satisfying $\dim(W) +\dim(Y_I) = km$.  Then there exists a Zariski-open subset $U \subset \oMat(n,k+m)$ such that:
\begin{enumerate}
\item for all $Z \in U$, we have $W \setminus E_Z$ is open and dense in $W$;
\item 
\begin{enumerate}
\item
either for all $Z \in U$  we have $\dim(\tau_Z(W)) < \dim(W)$ and $\tau_Z(W) \cap Y_I = \emptyset$, 
\item
or for all $Z \in U$ we have $\dim(\tau_Z(W))= \dim(W)$, the intersection $\tau_Z(W) \cap Y_I$ is transversal, and all intersection points lie in the locus inside $Z_\Gr(W \setminus E_Z)$ where the map $Z_\Gr|_{W \setminus E_Z}: W\setminus E_Z \to Z_\Gr(W \setminus E_Z)$ has fibers of cardinality exactly $d_Z$, where $d_Z$ is the degree of the map $Z_\Gr|_{W \setminus E_Z}$.  Furthermore, $d_Z$ is constant for all $Z \in U$.
\end{enumerate}
\item for all $Z \in U$, we have $W$ intersects $X^Z_I$ transversally, and all intersection points lie in $W \setminus E_Z$.
\end{enumerate}
\end{lemma}
\begin{proof}
In the following, we shall use the fact that a morphism between irreducible varieties is generically flat, and a morphism between smooth irreducible varieties is generically smooth.  (Here all varieties are over $\C$.)  Similar results are used throughout \cite{Kle}, and we refer the reader there for precise references.

For a fixed full-rank $Z_0$, by Theorem \ref{thm:Kleiman}, there exists an open subset $V \subset GL(n)$ such that $W$ and $gE_{Z_0}$ intersect properly for any $g \in V$.  By dimension considerations we will have $W \setminus gE_{Z_0}$ is open and dense in $W$ for any $g \in V$.  But $gE_{Z_0} = E_{gZ_0}$.  The map $g \mapsto gZ_0$ gives a surjective map from $GL(n)$ to $\oMat(n,k+m)$.  It follows that the image of $V$ contains a Zariski open subset $U_1 \subset \oMat(n,k+m)$.

Define $L \subset W \times U_1$ by
$$
L:= \{(X,Z) \mid X \notin E_Z\}.
$$
Obviously $L$ is an irreducible and open subset of $W \times U_1$, and the fiber of $L$ over $Z \in U_1$ is $W \setminus E_Z$.  Define $\mu: L \to \Gr(k,k+m)$ by $\mu(X,Z) = X\cdot Z \in \Gr(k,k+m)$.  Define
$$
S := (\mu \times \id)(L) = \{\mu(X,Z),Z \mid (X,Z) \in L\} \subset \Gr(k,k+m) \times U_1.
$$  Obviously, $S$ is an irreducible subset of $\Gr(k,k+m) \times U_1$.  We now assume that $\dim(S) = \dim(L)$ so that for $Z$ in a dense open subset of $U_1$, we have $\dim(\tau_Z(W)) =\dim(W)$, and we are in case (b) of part (2) of the Lemma.  (The case where $\dim(S) < \dim(L)$ is easier since one expects $\tau_Z(W)$ and $Y_I$ not to intersect.)   Let $d$ be the degree of the map $\mu \times \id: L \to S$.  For a dense open subset $S' \subset S$, the fiber $(\mu \times \id)^{-1}(s)$ for $s \in S'$ will have exactly $d$ points.

Let $\bar S$ denote the closure of $S$ in $\Gr(k,k+m) \times U_1$.  Over an open subset $U_2 \subset U_1$, the map $p:\bar S \to U_2$ will be flat, and the fiber over $Z \in U_2$ will be reduced and equal to $\tau_Z(W)$, and all these fibers will have the same dimension.  By shrinking $U_2$ if necessary, we may assume that for $Z \in U_2$, the fiber $\tau'_Z(W):=p^{-1}(Z) \cap S'$ is an open dense subset of $\tau_Z(W)$ that is contained in $Z_\Gr(W\setminus E)$, with the additional property that for $Y \in \tau'_Z(W)$, the fiber $Z_\Gr^{-1}(Y)$ has cardinality equal to exactly $d$.  Let $T =p^{-1}(U_2) \cap S' \subset S$.  By replacing $T$ by its smooth locus, and shrinking $U_2$ if necessary, we may in addition assume that $T$ is smooth.

%Let $T' = p|_T^{-1}(U_3) \subset T$.  

For $h \in G = GL(k+m)$, we have $E_Z = E_{Z\cdot h}$ and $X\cdot (Z \cdot h)=  h(X \cdot Z)$.  Thus if the above properties hold for $U_2 \subset \oMat(n,k+m)$ and $T \subset \Gr(k,k+m) \times \oMat(n,k+m)$, they also hold for $U_2 \cdot h$ and $T \cdot h$ (where $h\in G$ sends $(Y,Z)$ to $(hY,Z\cdot h)$).  So we may assume that $U_2$ and $T$ are closed under the $G$-action.

Consider the natural map $q: T \to \Gr(k,k+m)$.  Since $T$ is smooth and irreducible, $q$ is flat with smooth fibers over a dense subset $V$ of $\Gr(k,k+m)$.  The map $q$ commutes with the actions of $G$ on $T'$ and on $\Gr(k,k+m)$.  Since $G$ acts transitively on $\Gr(k,k+m)$, the translations $gV$ for $g \in G$ obviously cover $\Gr(k,k+m)$, and it follows that $q$ is flat over $\Gr(k,k+m)$.

Now we apply Lemma \ref{lem:Kleiman}(1,2) to the family 
\begin{center}
\begin{tikzpicture}[description/.style={fill=white,inner sep=2pt}]
\matrix (m) [matrix of math nodes, row sep=3em,
column sep=2.5em, text height=1.5ex, text depth=0.25ex]
{  & T& & Y^{\sm}_I \\
U_2& & \Gr(k,k+m) & \\ };
%\draw[double,double distance=5pt] (m-1-1) – (m-1-3);
\draw[->,font=\scriptsize]
(m-1-2) edge node[description] {$p$} (m-2-1);
\draw[->,font=\scriptsize]
(m-1-2) edge node[description] {$q$} (m-2-3);
\draw[->,font=\scriptsize]
(m-1-4) edge node[description] {$\iota$} (m-2-3);
\end{tikzpicture}
\end{center}
where $\iota: Y^{\sm}_I \to \Gr(k,k+m)$ is the inclusion of the smooth locus $Y^{\sm}_I$.  We deduce that there is a dense open subset $U_3 \subset U_2$ such that for each $Z \in U_3$, we have that $\tau'_Z(W)$ and $Y^{\sm}_I$ intersect transversally.  

To finish obtaining the statement of (2)(b), it remains to show that we can find $U_4 \subset U_3$ so that for $Z \in U_4$, we have $\tau_Z(W) \cap Y_I = \tau'_Z(W) \cap Y^{\sm}_I$.  That is, there are no intersection points in $\tau_Z(W) \setminus \tau'_Z(W)$ or $Y_I \setminus Y^{\sm}_I$.  To do so, we repeat the argument (using Lemma \ref{lem:Kleiman}(1)) for the family $(\bar S \setminus T) \to U_2$ and the inclusion $\iota: Y_I \to \Gr(k,k+m)$.  The typical fiber of $(\bar S \setminus T) \to U_2$ has lower dimension than $\tau_Z(W)$ since $\tau'_Z(W)$ is open dense in $\tau_Z(W)$.  Thus we expect $\tau_Z(W) \setminus \tau'_Z(W)$ not to intersect $Y_I$.  We deduce that there exists a dense open subset of $U_3$ where all the intersection points of $\tau_Z(W)$ and $Y_I$ lie in $\tau'_Z(W) \cap Y_I$.  Repeating the argument, we can also find a dense open subset of $U_3$ such that the intersection points of $\tau_Z(W)$ and $Y_I$ lie in $\tau_Z(W) \cap Y^{\sm}_I$.  Thus we can find $U_4 \subset \oMat(n,k+m)$ satisfying conditions (1) and (2) of the Lemma.

Finally, the condition (3) holds in an open subset $U' \subset \oMat(n,k+m)$: the argument here only requires applying Theorem \ref{thm:Kleiman}(1,2).  We then set $U:= U' \cap U_4$.  
\end{proof}

An irreducible subvariety $W \subset \Gr(k,n)$ of complex codimension $d$ has a cohomology class $[W] \in H^{2d}(\Gr(k,n))$, which must be non-zero.  Transverse intersections allow one to compute products in cohomology.

\begin{theorem}[{\cite[Appendix B]{Ful}}]\label{thm:Ful}
Let $X$ be a nonsingular variety.  Let $Y, Z \subset X$ be closed irreducible subvarieties.  Suppose $Y$ and $Z$ intersect transversally.  Then we have
$$
[Y] \cdot [Z] = [Y \cap Z]
$$
in the cohomology ring $H^*(X)$.
\end{theorem}
When $Y \cap Z$ is a finite set of $r$ (reduced) points, we have $[Y\cap Z] = r[\pt] \in H^*(X)$.

Let $E_\bullet$ be the standard flag in $\C^n$.  The cohomology ring $H^*(\Gr(k,n))$ vanishes in odd degrees, and the set $\{[X_I(E_\bullet)] \mid \codim(X_I) = d\}$ of Schubert classes forms a $\Z$-basis of $H^{2d}(\Gr(k,n))$. 

Recall that $H^*(\Gr(k,n))$ is isomorphic to the quotient of the ring $\Lambda$ of symmetric functions by an ideal $I_{k,n}$ (see \cite{Ful}).  Under this identification, we have
$$
[X_I] = s_{\lambda(I)}
$$
where $\lambda(I) = (i_k-k,i_{k-1}-(k-1),\ldots,i_1-1)$, and $s_\lambda$ denotes a Schur function.  Thus $[\Gr(k,n)] = s_{(0)}$ and $[\pt] = s_{(n-k)^k}$.  Let $\lambda^c$ denote the 180 degree rotation of the complement of $\lambda$ inside the $(n-k)^k$ rectangle.  Then $\lambda^c(J) = \lambda(I)$ where $I = J^c := \{(n+1) -j \mid j \in  J\}$.  Inside $H^*(\Gr(k,n))$, we have the equality 
\begin{equation}\label{eq:dual}
s_\lambda \, s_{\mu} = \begin{cases} 1 & \mbox{$\mu=\lambda^c$} \\
0 & \mbox{otherwise}
\end{cases}
\end{equation}
for $|\lambda|+|\mu| = k(n-k)$.  To summarize, a class $\sigma \in H^{2r}(\Gr(k,n))$ is determined by calculating $\sigma \, s_\mu$ for all $\mu$ satisfying $|\mu| = k(n-k)-r$.

Let $W \subset \Gr(k,n)$ be an irreducible subvariety.  Recall that in Section \ref{sec:intro}, we defined the truncation $\tau_{k+m}([W]) \in H^*(\Gr(k,k+m))$. 

\begin{proposition}\label{prop:main}
Let $U_I \in \Mat(n,k+m)$ denote the Zariski-open subset of Lemma \ref{lem:main} for $Y_I$, and let $U = \bigcap_I U_I$ where the intersection is over all $I$ such that $\dim(W) +\dim(Y_I) = km$.  
\begin{enumerate}
\item
If $\tau_{k+m}([W]) = 0$ then $\dim(\tau_Z(W)) < \dim(W)$ for all $Z \in U$.
\item
If $\tau_{k+m}([W]) \neq 0$ then for all $Z \in U$, we have $\dim(\tau_Z(W)) = \dim(W)$ and 
$$
[\tau_Z(W)] = \frac{1}{d} \tau_{k+m}([W])
$$
where $d$ is the degree of $Z_\Gr|_{W \setminus E_Z}$.
\end{enumerate}
\end{proposition}
\begin{proof}
Suppose $Z \in U$.  If $\tau_{k+m}([W]) \neq 0$ then by condition (3) of Lemma \ref{lem:main}, we can find $I \in \binom{[k+m]}{k}$ satisfying $\dim(W) +\dim(Y_I) = km$ so that $X_I^Z$ intersects $W \setminus E$ in a non-zero number of points.  The image of these points under $Z_\Gr$ lie in $\tau_Z(W) \cap Y_I$, and since this intersection is transverse, we must have $\dim(\tau_Z(W)) = \dim(W)$.  For each $I$, we have that $\tau_Z(W)$ intersects $Y_I$ transversally in a finite number of points $r_I$.  Also $W$ intersects $X_I^Z$ transversally in a finite number of points $s_I$, and from the conditions of Lemma \ref{lem:main}, we deduce that $s_I = d r_I$ from Lemma \ref{lem:inverse}.  Let $\lambda(I)^c$ be the complement of $\lambda(I)$ in the $k \times m$ rectangle.  It follows from Theorem \ref{thm:Ful} that the coefficient of $s_{\lambda(I)^c}$ in $[\tau_Z(W)]$ is equal to $r_I$ which is equal to $1/d$ times the coefficient of $s_{(\lambda(I)^c)^{+\ell}}$ in $[W]$.  Claim (2) follows.

Now suppose $\tau_{k+m}([W]) = 0$.  Then by a similar argument, we deduce that $\tau_Z(W)$ does not intersect any $Y_I$.  This is impossible if $\dim(\tau_Z(W)) = \dim(W)$ since $[\tau_Z(W)]$ has a non-zero cohomology class and the intersections $\tau_Z(W) \cap Y_I$ are transversal.  It follows that $\tau_{k+m}([W]) = 0$ implies that $\dim(\tau_Z(W)) < \dim(W)$.  Claim (1) follows.
\end{proof}

\section{Amplituhedron varieties and affine Stanley symmetric functions}
\label{sec:main}
\subsection{Affine Stanley symmetric functions}
Let $W_n$ denote the affine Coxeter group of type $A$, with generators $s_0,s_1,\ldots,s_{n-1}$, and relations 
\begin{align*}
s_i^2 &= 1 \\
s_i s_j &= s_j s_i & \mbox{if $|i-j| >1$}\\
s_i s_{i+1} s_i &= s_{i+1} s_i s_{i+1} 
\end{align*}
where all indices are taken modulo $n$.  The {\it length} $\ell(w)$ of $w \in W_n$ is the length of the shortest expression of $w$ as a product of the $s_i$.

An element $v \in W_n$ is called \emph{cyclically decreasing} if it has a reduced word $v = s_{i_1} s_{i_2} \cdots s_{i_k}$ such that $i_1, i_2, \ldots, i_k$ are distinct, and if both $i$ and $i+1$ occur then $i+1$ occurs before $i$.  For example, $s_4 s_3 s_1 s_0 s_6$ is cyclically decreasing if $n = 7$.  A cyclically decreasing factorization of $v$ is a factorization $v = v_1 v_2 \cdots v_r$ where $\ell(v) = \ell(v_1) + \ell(v_2) + \cdots + \ell(v_r)$ and each $v_i$ is cyclically decreasing.  For $v \in W_n$, we define the {\it affine Stanley symmetric function}
$$
\tF_v(x_1,x_2,\ldots) = \sum_{v = v_1 v_2 \cdots v_r} x_1^{\ell(v_1)} x_2^{\ell(v_2)} \cdots x_r^{\ell(v_r)}.
$$
In \cite{Lam} it is shown that $\tF_v$ is a symmetric function.

An affine permutation is a bijection $f: \Z \to \Z$ satisfying
\begin{enumerate}
\item $f(i+n) = f(i) + n$
\item $\sum_{i=1}^n (f(i) - i) = kn$
\end{enumerate}
A {\it $(k,n)$-bounded affine permutation} is an affine permutation satisfying
$$
i \leq f(i) \leq i+n.
$$
We denote the (finite) set of $(k,n)$-bounded affine permutations by $\Bound(k,n)$.  The group $W_n$ acts on the set of affine permutations on the right, with $s_i$ acting by swapping $f(i+rn)$ and $f(i+rn+1)$ for all $r \in \Z$.

Let $f_0: \Z \to \Z$ denote the bounded affine permutation given by $f_0(i) = i+k$.  Each bounded affine permutation $f$ has an expression as $f_0 s_{i_1} s_{i_2} \cdots s_{i_\ell}  = f_0 v$ for $v \in W_n$.  The length $\ell(f)$ of $f$ is declared to be equal to the length of $v$. We define $\tF_f : = \tF_v$.

\subsection{The cohomology class of a positroid variety}
Let $X \in \Gr(k,n)$.  Pick a $k \times n$ matrix representing $X$, with columns $v_1,v_2,\ldots,v_n$, and using $v_i = v_{i+n}$ we define $v_i$ for all $i \in \Z$.  Define a function $f_X: \Z \to \Z$ by
$$
f_X(i) = \min_{j \geq i}\left(v_i \in {\rm span}(v_{i+1},v_{i+2},\ldots,v_j) \right).
$$
Note that if $v_i = 0$ then $f_X(i) = i$.  It is not too hard to show \cite{KLS,Pos} that $f_X \in \Bound(k,n)$.

Let $f \in \Bound(k,n)$.  We define the open positroid variety
$$
\oPi_f:= \{X \in \Gr(k,n) \mid f_X = f\}
$$
and the positroid variety $\Pi_f:= \overline{\oPi_f}$.  We have a decomposition $\Gr(k,n) = \sqcup_{f \in\Bound(k,n)} \oPi_f$.  Let $\tF_f \in H^*(\Gr(k,n))$ be the image of the affine Stanley symmetric function in the quotient $\Lambda/I_{k,n} \simeq H^*(\Gr(k,n))$.  

\begin{theorem}[\cite{KLS}]\label{thm:KLS}
We have $[\Pi_f] = \tF_f \in H^*(\Gr(k,n))$.
\end{theorem}

We define the {\it truncated affine Stanley symmetric function} to be $\tau_{k+m}(\tF_f)$, where $\tF_f$ is thought of as an element of $H^*(\Gr(k,n))$.  

\subsection{The main theorem}
Let $f \in\Bound(k,n)$ be a $(k,n)$-bounded affine permutation, and $\Pi_f$ be the positroid variety labeled by $f$ \cite{Pos,KLS}.  For a general $Z$, we define
$$
Y_f := \overline{Z_\Gr(\Pi_f\setminus E_Z)}
$$
to be the closure of the image of $\Pi_f \setminus E_Z$ under $Z_\Gr$.  Obviously $Y_f$ depends on $Z$, but we will suppress this from the notation.  Define $Z_f$ to be $Z_f = Z_\Gr|_{\Pi_f \setminus E_Z}: (\Pi_f \setminus E_Z) \to Y_f$.

Our main result follows from Theorem \ref{thm:KLS} and Proposition \ref{prop:main} applied to $W = \Pi_f$.
\begin{theorem}\label{thm:main}
There exists a Zariski-open set $U \subset \Mat(n,k+m)$ such that
\begin{enumerate}
\item
if $\tau_{k+m}([\Pi_f]) = 0$ then $\dim(Y_f) < \dim(\Pi_f)$ for all $Z \in U$, and
\item
if $\tau_{k+m}([\Pi_f]) \neq 0$ then for all $Z \in U$, we have $\dim(Y_f) = \dim(\Pi_f)$ and 
$$
[Y_f] = \frac{1}{d} \tau_{k+m}([\Pi_f]) \in H^*(\Gr(k,k+m))
$$
where $d = \deg(Z_f)$ is the degree of $Z_f$, which is constant for all $Z \in U$.
\end{enumerate}
\end{theorem}

\begin{definition}
If $\dim Y_f = \dim \Pi_f$ for a general $Z$, then we declare $Y_f$ to be an \emph{amplituhedron variety}, and say that the affine permutation $f$ has \emph{kinematical support}.
\end{definition}

From now on, $Y_f$, $Z_f$ and $\deg(Z_f)$ will always refer to an amplituhedron variety, the corresponding map and its degree, for $Z \in U$. 

\begin{corollary}
Let $f$ be a $(k,n)$-bounded affine permutation.  Then $f$ has kinematical support if and only if for some partition $\lambda$ satisfying $\ell^k \subseteq \lambda \subseteq (n-k)^k$, the coefficient of $s_\lambda$ in the affine Stanley symmetric function $\tF_f$ is non-zero.
\end{corollary}

When $\dim(\Pi_f) = km$, this says that $f$ has kinematical support if and only if $s_{\ell^k}$ appears in $\tF_f$ with non-zero coefficient.  In this case, the coefficient of $s_{\ell^k}$ in $\tF_f$ can be computed using the affine Pieri rule for the flag variety (see Remark \ref{rem:pieri}).

\begin{example}
Let $k=2$, $m = 4$, and $n = 8$. Suppose $f = [4,3,6,5,8,7,10,9] \in Bound(2,8)$, which can be written as $f = f_0 s_1 s_3 s_5 s_7$.  Then from the definitions we have $\tF_f = (\sum_{i=1}^\infty x_i)^4$.  The coefficient of $s_{(2,2)}$ in $\tF_f$ is equal to $2$.  So $f$ has kinematical support and the map $Z_f: (\Pi_f \setminus E_Z) \to Y_f$ has degree $2$.

In a similar manner we can easily produce maps $Z_f$ of arbitrarily high finite degree.
\end{example}

\begin{remark}
Each positroid variety $\Pi_f$ comes from a canonical meromorphic top-form $\omega_{\Pi_f}$.  If $f$ has kinematical support then $Z_f$ is generically finite and we can define a canonical form $\omega_{Y_f}$ by pushing forward the canonical form $\omega_{\Pi_f}$ of the positroid variety.  See \cite{LamCan}.
\end{remark}

\begin{remark}
If $m = 4$ and $\dim Y_f = \dim \Pi_f = 4k = \dim \Gr(k,k+4)$, then our notion of kinematical support essentially agrees with the notion from the theory of scattering amplitudes \cite{ABCGPT}, though we caution that the work \cite{ABCGPT} is mostly set in ``momentum space", while the present work is set in ``momentum-twistor space".  Physically, it is clear that when considering amplituhedron cells of dimension $4k$, one should restrict to those cells with kinematical support.  %For cells of dimension smaller than $4k$, it is not clear to me if this notion has a simple physical meaning.
\end{remark}

\subsection{Degree of $Z_f$}

For the cells $Y_f$ of dimension $km$ that are used to triangulate the amplituhedron, we have an easy criterion for the degree.  For $m =4$, this is presumably the same combinatorial criterion discussed in \cite{ABCGPT}, after translating from ``momentum space" to ``momentum-twistor space".

\begin{proposition}
Suppose $\dim(\Pi_f) = km$.  Then the degree of $Z_f$ is the coefficient of $s_{\ell^k}$ in $\tF_f$, if this coefficient is positive.  If this coefficient is 0, then $f$ does not have kinematical support.
\end{proposition}

\begin{problem}
Let $k$ and $m$ be fixed, and allow $n$ to vary.  Is there a uniform bound on $d_f$ for all $f$ with kinematical support?
\end{problem}

\begin{remark}\label{rem:pieri}
Suppose $f$ has kinematical support and $\dim(\Pi_f) = km$.  The Pieri rule for the affine flag manifold conjectured in \cite{LLMS}, and proved in \cite{Lee} can be used to give a manifestly positive formula for the degree $d_f = \deg(Z_f)$.  

Specifically, the (dual) affine Pieri rule gives an identity
$$
e_k \tF_f = \sum_g c_{k,f}^g \tF_g
$$
where the nonnegative numbers $c_{k,f}^g$ count objects called {\it strong strips} \cite{LLMS}.  Now, we have
$$
[s_{\ell^k}]\tF_f = [s_{(n-k)^k}] (e_k)^m \tF_f
$$
where $[s_{\ell^k}]\tF_f$ denotes the coefficient of $s_{\ell^k}$ when $\tF_f$ is expanded in terms of Schur functions.  By \cite[Theorem 7.8]{KLS}, $$
[s_{(n-k)^k}] \tF_g = \begin{cases} 1 & \mbox{$g \in \Bound(k,n)$ and $\ell(g) = k(n-k)$} \\
0 & \mbox{otherwise.} 
\end{cases}
$$
Thus $d_f = [s_{\ell^k}]\tF_f$ can be obtained by counting iterated strong strips. 
\end{remark}

The following result also gives a criterion for the degree of $Z_f$ to be 1.

\begin{proposition}
Suppose $f$ has kinematical support and $\tF_f = \sum_\lambda c_\lambda s_\lambda \in H^*(\Gr(k,n))$.  Let $c = \gcd_{\mu \subset m^k}(c_{\mu^{+\ell}})$.  Then the degree of $Z_f$ divides $c$.  In particular, if $c = 1$, then $\deg(Z_f)= 1$.
\end{proposition}

For the next result, we will use the following version of Zariski's main theorem.

\begin{theorem}\label{thm:Zariski}
If $Y$ is a quasi-compact separated scheme and $f: X \to Y$ is a separated, quasi-finite, finitely presented morphism then there is a factorization into $X \to Z \to Y$, where the first map is an open immersion and the second one is finite.
\end{theorem}

The following result roughly says that taking boundaries reduces the degree.  So the intuition is that lower-dimensional cells tend to have smaller degree.  Denote by $\partial \Pi_f$ the {\it boundary} of a positroid variety $\Pi_f$.  This is the union of all positroid varieties $\Pi_{f'} \subset \Pi_f$ of strictly lower dimension.  For more details on the closure partial order of positroid varieties see \cite{KLS,Pos}.

\begin{proposition}
Suppose $\Pi_{f'} \subset \partial \Pi_f$ and both $f'$ and $f$ have kinematical support.  Then $\deg(Z_{f'}) \leq \deg(Z_f)$.
\end{proposition}

\begin{proof}
By applying Lemma \ref{lem:main} to both $W = \Pi_f$ and $W = \Pi_{f'}$, we see that we may assume that we are considering $Z \in \Mat(n,k+m)$ such that $\Pi_{f} \setminus E_Z$ is dense in $\Pi_f$ and $\Pi_{f'} \setminus E_Z$ is dense in $\Pi_{f'}$.    We may suppose that $Z_f$ has degree $d_f$ and $Z_{f'}$ has degree $d_{f'}$ where both maps are dimension-preserving.

Let $V \subset (\Pi_f \setminus E_Z)$ consist of points $X \in \Pi_f \setminus E_Z$ where $Z_f^{-1}(Z_f(X))$ is finite.  Since fiber dimension is upper semicontinuous on the source, the set $V$ is open in $\Pi_f\setminus E_Z$.  But $Z_{f'} = Z_f|_{\Pi_{f'} \setminus E_Z}$ so $V \cap (\Pi_{f'} \setminus E_Z)$ is open in $\Pi_{f'} \setminus E_Z$ as well. 

By assumption $Z_f|_V$ is quasi-finite, so by Theorem \ref{thm:Zariski}, we have a factorization of $Z_f|_V$ as $V \to S \to Y_f$, where $V \to S$ is an open immersion and $S \to Y_f$ is finite.  Clearly, $S \to Y_f$ has degree $d_f$.  It follows that the typical fiber of $Z_f|_V$ has exactly $d_f$ points, and every fiber of $Z_f|_V$ has $\leq d_f$ points.  In particular, the typical fiber of $Z_{f'}$ has $\leq d_f$ points.  Thus $d_{f'} \leq d_f$.
\end{proof}

A similar argument gives

\begin{proposition}
Suppose $\Pi_{f'} \subset \partial \Pi_f$ and $f$ does not have kinematical support.  Then $f'$ does not have kinematical support.
\end{proposition}

\subsection{Application to the amplituhedron}
The totally nonnegative part $\Gr(k,n)_{\geq 0}$ \cite{Pos} of the real Grassmannian is the locus of points $X \in \Gr(k,n)(\R)$ representable with nonnegative (real) Pl\"ucker coordinates $\Delta_I(X)$.  The totally nonnegative part of $\Pi_f$ is defined to be $(\Pi_f)_{\geq 0} := \Pi_f \cap \Gr(k,n)_{\geq 0}$.

We say that $Z$ is {\it positive} if all $(k+m) \times (k+m)$ minors are positive, and all entries are real.  If $Z$ is positive and general (that is, $Z$ belongs to the Zariski-dense set $U$ of Theorem \ref{thm:main}), we define the TNN part of $Y_f$ to be
$$
(Y_f)_{\geq 0} := Z_\Gr((\Pi_f)_{\geq 0}).
$$
As shown in \cite{AT}, it is not difficult to see that in this case $(\Pi_f)_{\geq 0}$ does not intersect $E_Z$.

\begin{proposition}
Suppose $f$ has kinematical support.  Then $\dim_\R((Y_f)_{\geq 0}) = \dim(Y_f)$.  %Furthermore, $Z_f$ is generically $d_f$ to $1$ when restricted to $(\Pi_f)_{\geq 0}$.
\end{proposition}
\begin{proof}
It is known \cite{Pos} that $(\Pi_f)_{\geq 0}$ has real dimension equal to the complex dimension of $\Pi_f$.  Since $\Pi_f$ is irreducible (see \cite{KLS}), it follows that $(\Pi_f)_{\geq 0}$ is Zariski-dense in $\Pi_f$.  It follows that $(Y_f)_{\geq 0}$ is Zariski-dense in $Y_f$ and thus $\dim_\R((Y_f)_{\geq 0}) = \dim(Y_f)$.
%When $f$ has kinematical support, the map $Z_f$ has finite fibers of size $d$ over a Zariski-open subset of $\Pi_f \setminus E$.  It follows that $Z_f$ is generically $d_f$ to $1$ when restricted to $(\Pi_f)_{\geq 0}$.  It also follows that $\dim_\R((Y_f)_{\geq 0}) = \dim(Y_f)$.
\end{proof}

Suppose $f$ has kinematical support and $Z_f$ has degree $d_f$, and assume that $Z$ is positive.  While the map $Z_f$ has degree $d_f$, it is not the case that the map $Z_f$ restricted to $(\Pi_f)_{\geq 0}$ is generically $d_f$ to $1$.  It is an interesting problem to understand the geometry of the map $Z_f$ when restricted to the real points $\Pi_f(\R)$ or totally nonnegative points $(\Pi_f)_{\geq 0}$.

\section{Some further directions}\label{sec:other}
\subsection{Monomial description of truncated affine Stanley symmetric functions}
Since the truncated affine Stanley symmetric function $\tau_{k+m}(\tF_f)$ is Schur-positive, it is also monomial-positive.  However, it is not clear which monomials in the definition of $\tF_f$ actually contribute to the truncation.
\begin{problem}
Find a direct combinatorial description of the monomial expansion of $\tau_{k+m}(\tF_f)$.
\end{problem}

Presumably this involves selecting {\it some} of the cyclically decreasing factorizations of $f$ to contribute to $\tau_{k+m}(\tF_f)$.

\subsection{Non-generic maps $Z$}
Our results only apply to generic $Z \in \oMat(n,k+m)$.  However, from the point of view of convex geometry, it is interesting to consider non-generic maps.  Specifically, when $k = 1$, the totally nonnegative Grassmannian $\Gr(1,n)_{\geq 0}$ is a simplex embedded in projective space, and thus {\it any} polytope $P$ can be expressed as the image $Z_\Gr(\Gr(1,n)_{\geq 0})$ for some choice of $Z$.  When $Z$ is positive, $P$ will be a cyclic polytope (see \cite{Stu} for a related result).

\begin{problem}
Compute the cohomology class $[\overline{Z_\Gr(\Pi_f \setminus E_Z)}] \in H^*(\Gr(k,k+m))$ for all $Z \in \oMat(n,k+m)$.
\end{problem}

A related, possibly easier, problem is the following. 
\begin{problem}
What is the cohomology class of $X_I(F_\bullet) \cap X_J(G_\bullet)$ when the two Schubert varieties are not in generic position?
\end{problem}

\end{document}